\documentclass[a4paper,twoside,11pt,leqno]{article}

\usepackage{amsthm,amsmath,natbib,amsfonts,graphicx}
\RequirePackage[colorlinks]{hyperref}

\def\authors{George Lowther}
\def\runauthor{\authors}
\def\head{Fitting Martingales To Given Marginals}
\def\runhead{Martingale Marginals}
\def\keywords{Martingale, Strong Markov, Diffusion, Marginal distributions.}
\def\classifications{60J60, 60J25, 60G44, 60H10.}

\makeatletter
\def\@evenhead{\thepage\hfill{\small\MakeUppercase{\runauthor}}\hfill}
\def\@oddhead{\hfill{\small\MakeUppercase{\runhead}}\hfill\thepage}
\def\blfootnote{\xdef\@thefnmark{}\@footnotetext}
\makeatother

\newcommand{\halfplane}{\reals_+\times\reals} 
\newcommand{\ohalfplane}{\reals_{>0}\times\reals} 
\newcommand{\reals}{\mathbb{R}} 

\newcommand{\E}[1]{\mathbb{E}\left[#1\right]} 
\newcommand{\setsF}{\mathcal{F}} 
\newcommand{\setsG}{\mathcal{G}} 
\newcommand{\pd}[1]{_{\!,_{#1}}}
\newcommand{\dlimint}[4]{\int_{#1}^{#2\!\!\!}\!\int_{#3}^{#4}} 
\newcommand{\dblint}{\int\!\!\!\int} 
\newcommand{\nat}{\mathbb{N}}
\newcommand{\Prob}[1]{\mathbb{P}\left(#1\right)} 
\newcommand{\salg}{$\sigma$-algebra} 

\newcommand{\cp}{{\rm CP}}

\newcommand{\PP}{\mathbb{P}}
\newcommand{\QQ}{\mathbb{Q}}
\newcommand{\acd}{ACD}
\newcommand{\support}[1]{{\rm Supp}(#1)}
\newcommand{\msupport}[1]{{\rm MSupp}(#1)}
\newcommand{\EP}[2]{\mathbb{E}_{#1}\left[#2\right]}   
\newcommand{\D}{{\rm\bf D}}
\newcommand{\measures}[1]{{{\mathcal M}(#1)}}
\newcommand{\nicefunc}{\mathcal{D}}
\newcommand{\nicefuncK}{\mathcal{D}_{\rm K}}
\newcommand{\cadlag}{c\`adl\`ag}
\newcommand{\zcqv}{{z.c.q.v.}}
\newcommand{\Dom}{\mathcal D}
\newtheorem{definition}{Definition}[section]
\newtheorem{theorem}[definition]{Theorem}
\newtheorem{lemma}[definition]{Lemma}
\newtheorem{corollary}[definition]{Corollary}

\numberwithin{equation}{section}

\begin{document}

\title{\head}
\blfootnote{\emph{Key Words}: \keywords}
\blfootnote{\emph{AMS 2000 Classification}: \classifications}
\author{\authors}
\date{}
\maketitle
\thispagestyle{empty}

\begin{abstract}
We consider the problem of finding a real valued martingale fitting specified marginal distributions. For this to be possible, the marginals must be increasing in the convex order and have constant mean. We show that, under the extra condition that they are weakly continuous, the marginals can always be fitted in a unique way by a martingale which lies in a particular class of strong Markov processes.

It is also shown that the map that this gives from the sets of marginal distributions to the martingale measures is continuous. 
Furthermore, we prove that it is the unique continuous method of fitting martingale measures to the marginal distributions.
\end{abstract}

\maketitle

\section{Introduction}

We consider the problem of finding real valued martingales fitting given marginal distributions and show that, by restricting to a certain class of strong Markov processes, it can be done in a unique way.  It is furthermore shown that this is the unique continuous method of matching any specified marginals by martingales.

The existence of martingales with specified marginals has been previously studied by many authors. In particular, Strassen \citep{Strassen} showed in 1965 that if $(\mu_n)_{n\in\nat}$ is a sequence of probability measures on the real numbers which have constant mean and are increasing in the convex order, then there is a martingale $(X_n)_{n\in\nat}$ such that the law of $X_n$ is $\mu_n$.
The property that $\mu_n$ is increasing in the convex order simply means that $\mu_n(f)$ is increasing in $n$ for every increasing convex function $f$, and the necessity of this condition follows easily from Jensen's inequality. This result was extended by Kellerer \citep{Kellerer} in 1972 to the case where the marginal distributions $\mu_t$ and the martingale $X_t$ are indexed by time $t$ in $\reals_+$. It was also shown that $X$ can always be chosen to be Markov.

More recently, this problem has been investigated in the context of pricing financial derivatives, where knowledge of the prices of vanilla call and put options provides an implied distribution for the underlying asset price at future times.
For example, assuming zero interest rates (for simplicity) the local volatility model constructs the asset price process as a solution to the stochastic differential equation
\begin{equation}\label{eqn:LV SDE}
dS_t=S_t\sigma(t,S_t)\,dB_t,
\end{equation}
where $B$ is a Brownian motion, $S$ is the asset price and $\sigma(t,x)$ is the local volatility. Then, as is well known (see \citep{Derman} and \citep{Dupire2}), if $C(t,x)$ is the price of a vanilla call with strike price $x$ and maturity $t$, the implied probability density of $S_t$ is $\partial^2C/\partial x^2$ and the local volatilities can be recovered from the following forward equation
\begin{equation*}
\frac{\partial}{\partial t} C(t,x) = \frac{1}{2} x^2\sigma^2(t,x)\frac{\partial^2}{\partial x^2}C(t,x).
\end{equation*}
Alternative methods of matching the implied marginal distributions have been considered, such as jump-diffusions in \citep{Andersen}, stochastic volatility in \citep{ImpliedPrice} and models based on L\'evy processes in \citep{LevyModels}. Also, \citep{Yor} gives several constructions, including Skorokhod embedding and time-changed Brownian motion methods.

In this paper, we provide a general way of matching marginal distributions under very mild constraints. Other than the necessary conditions of having constant mean and being increasing in the convex order, the only further constraint placed on the marginals is that they be weakly continuous. That is, if $t_n\rightarrow t$ then $\mu_{t_n}(f)\rightarrow\mu_t(f)$ for every continuous and bounded function $f$. It is shown that such marginals can be fitted in a unique way by a certain class of strong Markov martingales. As this class includes all martingale diffusions, the solution will coincide with the local volatility model when it applies. However, for marginals which are either not smooth or don't have strictly positive densities, different types of solutions are obtained which cannot be described by an S.D.E. such as (\ref{eqn:LV SDE}). For example, jump processes and singular diffusions, as described in Section \ref{sec:examples}.

We also show that that the resulting map from the sets of marginals to the martingales is continuous. So, a small change to marginal distributions results in only a small change to the martingale measure matching these marginals.

Furthermore, it is shown in Theorem \ref{thm:acd martingale method is unique cts} that not only is our method of fitting the marginals continuous, but it is the only possible continuous method. Consequently, any alternative approach (e.g., those described by \citep{Andersen}, \citep{ImpliedPrice} and \citep{LevyModels}) must either fail to fit, or very closely approximate, certain marginal distributions, or small changes in the marginals would lead to big changes in the resulting martingale measure.

Let us now define the types of processes to be considered, which should include all continuous and strong Markov processes. However, there are some marginal distributions which cannot be matched by any continuous process. For example, if $\Prob{0<X_t<1}=0$ for all times $t$ and $\Prob{X_t\le 0}$ decreases in $t$, then there must be a positive probability that $X$ jumps from below $0$ to above $1$.
For this reason, we relax the continuity condition to obtain the following class of processes.

\begin{definition}\label{defn:acd}
Let $X$ be a real valued stochastic process. Then,
\begin{enumerate}
\item $X$ is \emph{strong Markov} if for every bounded, measurable $g:\reals\rightarrow\reals$ and every $t\in\reals_+$ there exists a measurable $f:\halfplane\rightarrow\reals$ such that
\begin{equation*}
f(\tau,X_\tau) = \E{g(X_{\tau+t})\mid \setsF_\tau}
\end{equation*}
for every finite stopping time $\tau$.
\item\label{defn:ac} $X$ is \emph{almost-continuous} if it is \cadlag, continuous in probability and given any two independent \cadlag\ processes $Y,Z$ each with the same distribution as $X$ and for every $s<t\in\reals_+$ we have
\begin{equation*}
\Prob{Y_s<Z_s,\, Y_t>Z_t{\rm\ and\ }Y_u\not= Z_u\textrm{ for every }u\in(s,t)}=0.
\end{equation*}
\item $X$ is an \emph{almost-continuous diffusion} if it is strong Markov and almost-continuous.
\end{enumerate}
\end{definition}
In \citep{Lowther2} it was shown that almost-continuous diffusions arise when taking limits of continuous diffusions in the sense of finite-dimensional distributions. Note that condition \ref{defn:ac} is equivalent to saying that $Y-Z$ cannot change sign without passing through zero, which is clearly true for continuous processes by the intermediate value theorem. In what follows, we often abbreviate `almost-continuous diffusion' to \acd.

An alternative way of representing the marginal distributions $\mu_t$ which we make use of is through the function $C(t,x)=\int(y-x)_+\,d\mu_t(y)$. The property that $\mu_t$ is increasing in the convex order is then equivalent to $C(t,x)$ being an increasing function of $t$. Furthermore, the distribution functions are easily recovered from $\mu_t((-\infty,x])=1+C\pd{2}(t,x+)$. This leads to the following space of functions.

\begin{definition} Let $\cp$ be the set of functions $C:\halfplane\rightarrow\reals$ such that
\begin{enumerate}
\item $C(t,x)$ is convex in $x$ and continuous and increasing in $t$.
\item $C(t,x)\rightarrow 0$ as $x\rightarrow\infty$, for every $t\in\reals_+$.
\item There exists a real number $a$ such that $C(t,x) + x \rightarrow a$ as $x\rightarrow-\infty$ for every $t\in\reals_+$.
\end{enumerate} 
\end{definition} 

Continuity of $C(t,x)$ in $t$ is just requiring the marginals $\mu_t$ to be weakly continuous in $t$, and the third condition is equivalent to them having a constant mean. The property that a process $X$ has marginals consistent with some $C\in\cp$ can be expressed as
\begin{equation}\label{eqn:unique measure fitting marginals}
C(t,x)=\E{(X_t-x)_+}
\end{equation}
 and, conversely, if $X$ is a martingale which is continuous in probability then $C$ given by (\ref{eqn:unique measure fitting marginals}) will be in the space $\cp$. The notation used here is borrowed from the financial interpretation where $C(t,x)$ are call prices, with maturity $t$ and strike $x$, although we just make use of $C\in\cp$ as convenient representations of martingale marginals both in the statements of the main results below and in the proofs later.

We use the space of \cadlag\ real valued processes (Skorokhod space) with coordinate process $X$ on which to represent martingale measures.
\begin{equation*}\begin{split}
&\D=\left\{\textrm{\cadlag\ functions }\omega:\reals_+\rightarrow\reals\right\},\\
&X:\reals_+\times\D\rightarrow\reals,\ (t,\omega)\mapsto X_t(\omega)\equiv \omega(t),\\
&\setsF=\sigma\left(X_t:t\in\reals_+\right),\\
&\setsF_t=\sigma\left(X_s:s\in[0,t]\right).
\end{split}\end{equation*}
Then, $(\D,\setsF)$ is a measurable space and $X$ is a \cadlag\ process adapted to the filtration $\setsF_t$.
The existence and uniqueness of the martingale measure fitting given marginals is now stated,

\begin{theorem}\label{thm:unique measure fitting marginals}
For any $C\in\cp$ there exists a unique measure $\PP$ on $(\D,\setsF)$ under which $X$ is an \acd\ martingale and (\ref{eqn:unique measure fitting marginals}) is satisfied.
\end{theorem}

See sections \ref{sec:existence} and \ref{sec:uniqueness} for the proof of this result, which involves a weak compactness argument to construct the measure and applies a result from \citep{Lowther2} concerning limits of almost-continuous diffusions. Then, a backward equation developed in \citep{Lowther4} is applied to show uniqueness.

Given any $C\in\cp$, the notation $\PP_C$ will be used for the unique \acd\ martingale measure matching the marginal distributions given by $C$. This defines a map $C\mapsto\PP_C$, which we shall  show is continuous under the appropriate topologies.

Let $\measures{\D}$ be the set of probability measures on $(\D,\setsF)$. A sequence $(\PP_n)_{n\in\nat}$ in $\measures{\D}$ converges to $\PP$ in the sense of finite-dimensional distributions if and only if
$\EP{\PP_n}{Z}\rightarrow\EP{\PP}{Z}$
for every random variable $Z$ of the form
\begin{equation}\label{eqn:rv Z for finite dim}
Z=f(X_{t_1},\ldots,X_{t_m})
\end{equation}
for $t_1,\ldots,t_m\in\reals_+$ and continuous bounded $f:\reals^m\rightarrow\reals$.

We use the topology of pointwise convergence on $\cp$, so $C_n\rightarrow C$ if and only if $C_n(t,x)\rightarrow C(t,x)$ for all $(t,x)\in\halfplane$. Note that this is slightly stronger than weak convergence of the marginal distributions, which would be equivalent to convergence of $C_n(t,x)-C_n(t,y)$ to $C(t,x)-C(t,y)$.
The continuity result for the map from the marginals to the martingale measures is as follows.

\begin{theorem}\label{thm:marginals to martingales is cts}
For every $C\in\cp$ denote the unique \acd\ martingale measure given by Theorem \ref{thm:unique measure fitting marginals} by $\PP_C$.
Then, the function
\begin{equation*}
\cp\rightarrow\measures{\D},\ C\mapsto\PP_C
\end{equation*}
is continuous, under pointwise convergence on $\cp$ and convergence in the sense of finite-dimensional distributions on $\measures{\D}$.
\end{theorem}
So, given any sequence $C_n\in\cp$ converging pointwise to $C\in\cp$ then
$\EP{\PP_{C_n}}{Z}\rightarrow\EP{\PP_C}{Z}$
for every random variable $Z$ of the form (\ref{eqn:rv Z for finite dim}). The proof of this is left until Section \ref{sec:uniqueness}.

Not only is the \acd\ martingale measure fitting the marginal distributions uniquely defined, but it is also the \emph{only} way of fitting the marginals in a continuous way, as the following result states. Here, we again use the topology of pointwise convergence on $\cp$ and convergence in the sense of finite-dimensional distributions on $\measures{\D}$.

\begin{theorem}\label{thm:acd martingale method is unique cts}
Suppose that we have a continuous map from a dense subset $S$ of $\cp$ to the martingale measures
\begin{equation*}
S\rightarrow\measures{\D},\ C\mapsto\mathbb{Q}_C
\end{equation*}
such that for every $C\in S$ the equality $\EP{\QQ_C}{(X_t-x)_+}=C(t,x)$ is satisfied.
Then $\QQ_C=\PP_C$.
\end{theorem}
In particular this shows that the choice of the class of almost-continuous diffusions used to fit the marginals is not arbitrary, but was in fact forced upon us.
The proof of Theorem \ref{thm:acd martingale method is unique cts} is left until Section \ref{sec:extremal}, where the idea is that there are certain marginal distributions for which there is only one possible martingale measure. These correspond to extremal elements of $\cp$, and form a dense subset.

We finally note that the fact that $C(t,x)$ is continuous and monotonic in both $t$ and $x$ for every $C\in\cp$ implies that pointwise convergence is the same as locally uniform convergence, and the topology is given by the metric
\begin{equation}\label{eqn:cp metric}
d(C_1,C_2) = \sup\left\{|C_1(t,x)-C_2(t,x)|\wedge 2^{-|x|-t}:(t,x)\in\halfplane\right\}.
\end{equation}
So, we are justified in only considering limits of sequences (rather than generalized sequences) in the explanations and proofs of theorems \ref{thm:marginals to martingales is cts} and \ref{thm:acd martingale method is unique cts}.

\section{Examples}
\label{sec:examples}

In this section we mention some examples to demonstrate the kinds of processes which can result from different properties of the marginal distributions.

\subsection{Continuous diffusions}
\label{sec:examples:cts diff}

If $C(t,x)$ is strictly convex in $x$ for every $t>0$, then the support of $X_t$ under the measure given by Theorem \ref{thm:unique measure fitting marginals} will be all of $\reals$ and, consequently, $X$ will be a continuous process (see \citep{Lowther2} Lemma 1.4).
If, furthermore, $C$ is twice continuously differentiable then it can be shown that $X$ is a solution to the stochastic differential equation
\begin{equation}\label{eq:SDE example}
dX_t=\sigma(t,X_t)\,dB_t,
\end{equation}
for a Brownian motion $B$ and with $\sigma$ given by the forward equation
\begin{equation}\label{eq:fwd PDE example}
\frac{\partial}{\partial t}C(t,x)=\frac{1}{2}\sigma(t,x)^2\frac{\partial^2}{\partial x^2}C(t,x).
\end{equation}
Then, if $\sigma(t,x)$ is H\"older continuous of order $1/2$, the Yamada-Watanabe theorem (\citep{Rogers} V, Theorem 40.1) says that (\ref{eq:SDE example}) uniquely determines the law of $X$. This is the familiar situation covered by the local volatility model, and widely employed in finance (see \citep{Dupire2}).

\subsection{Jump processes}

Now suppose that the supports of the marginal distributions are contained in the set of integers. Then, the process $X$ must be an integer valued pure jump process.
Suppose furthermore that $\Prob{X_t=n}>0$ for every $t>0$ and integer $n$. Then, the almost-continuous property says that $X$ cannot jump past any integer values, so can only jump between successive integers. Therefore, $X$ must be a piecewise constant process with jump sizes $\pm 1$. For example, it could be a symmetric Poisson process (i.e., the difference of two standard Poisson processes).

More generally, jump processes can arise whenever the supports of the marginal distributions are not connected intervals. Consider, for example, a smooth $C(t,x)$ in $\cp$ which is strictly convex in $x$, so that the conditions considered in Section \ref{sec:examples:cts diff} are satisfied. Then define $\tilde C$ by
\begin{equation*}
\tilde C(t,x) = \left\{
\begin{array}{ll}
C(t,x),&\textrm{if }x\ge 1\textrm{ or }x\le 0,\\
xC(t,1)+(1-x)C(t,0),&\textrm{if }0<x<1.
\end{array}
\right.
\end{equation*}
The corresponding marginal distributions then assign zero probability to the interval $(0,1)$. Under the resulting martingale measure $\PP_{\tilde C}$, the process $X$ will behave like a continuous diffusion satisfying the SDE (\ref{eq:SDE example}) whenever $X>1$ or $X<0$. However, the points $\{0,1\}$ will act like reflecting barriers, compensated by $X$ sometimes jumping across the interval $(0,1)$.

\subsection{Singular diffusions}

Now suppose that $C(t,x)$ is strictly convex in $x$ for every $t>0$ so that, as in section \ref{sec:examples:cts diff}, we can conclude that $X$ is continuous under the associated \acd\ martingale measure.
If, however, $C(t,x)$ is not twice differentiable in $x$ then the marginal distributions will not be continuous with respect to the Lebesgue measure, and $X$ will not satisfy a stochastic differential equation such as (\ref{eq:SDE example}).

For example, suppose that $C\in\cp$ satisfies all of the properties considered in section \ref{sec:examples:cts diff} and define
\begin{equation*}
\tilde C(t,x) = \frac{1}{2}C(t,x) + \frac{1}{2}\max(-x,0).
\end{equation*}
Note that $\tilde C(t,x)$ is not differentiable at $x=0$ and the corresponding marginal distributions have an atom at $0$. Under the \acd\ martingale measure $\PP_{\tilde C}$, the process $X$ will behave like a continuous diffusion satisfying (\ref{eq:SDE example}) away from $0$. However, $\PP_{\tilde C}(X_t=0)=1/2$, so $X$ is sticky at $0$, spending a positive time there. If, furthermore, $C(t,x)$ is strictly increasing in $t$ then $X$ will not be constant over any time intervals.

Similarly, it is not difficult to construct marginal distributions so that $X_t$ is rational with probability $1$ and with support equal to $\reals$, resulting in continuous processes spending almost all their time in the rational numbers. This is the case with the Feller-McKean diffusion (\citep{Rogers1} III.23), and similar situations can arise as a limit of random walks with randomly generated rates (see \citep{Fontes}).

\section{Existence}
\label{sec:existence}

We show how \acd\ martingales can be constructed with specified marginal distributions by taking limits of processes which match the marginals at finite sets of times. A weak compactness argument is used to prove existence of the limit.

For any set $S\subseteq\reals_+$, let $\reals^S$ consist of the real valued functions on $S$. We consider $\reals^S$ as a topological space using the topology of pointwise convergence, and denote its Borel \salg\ by $\setsF^S$.
The weak topology on the probability measures on $(\reals^S,\setsF^S)$ is the topology generated by the maps $\PP\mapsto\EP{\PP}{f}$ for all real valued continuous and bounded functions $f$ on $\reals^S$.
We denote the coordinate process on $\reals^S$ by $X^S_t$,
\begin{equation*}
X^S:S\times\reals^S\rightarrow\reals,\ (t,\omega)\mapsto X^S_t(\omega)\equiv \omega(t),
\end{equation*}
which has natural filtration $(\setsF^S_t)_{t\in\reals_+}$ given by,
\begin{equation*}
\setsF^S_t=\sigma\left(X^S_s:s\in S,\,s\le t\right).
\end{equation*}
Then, for any measure $\PP$ on $(\D,\setsF)$ we use $\PP^S$ to denote the measure on $(\reals^S,\setsF^S)$ obtained from the law of $X_t$ under $\PP$ with $t$ restricted to $S$.

In particular, if $S$ is countable then $\reals^S$ is a Polish space, as it has a countable dense subset consisting of those $\omega$ such that $\omega(t)$ is rational for all $t\in S$ and zero for all but finitely many $t$, and the topology is given by a complete metric
\begin{equation*}
d(\omega,\omega^\prime)=\sum_n 2^{-n}\min(|\omega(s_n)-\omega^\prime(s_n)|,1),
\end{equation*}
where $S=\{s_1,s_2,\ldots\}$.

Furthermore, a sequence of probability measures $\PP_n$ on $(\D,\setsF)$ converges to $\PP$ in the sense of finite-dimensional distributions if and only if $\PP_n^S\rightarrow\PP^S$ weakly for every finite subset $S$ of $\reals_+$.

We now prove the result that we need in order to be able to find limits of sequences of martingale measures.
The idea here is to use weak compactness in order to pass to convergent subsequences.

A set $P$ of probability measures on a Polish space is said to be tight if for every $\epsilon>0$ there exists a compact set $C$ with $\Prob{C}<\epsilon$ for all $\PP\in P$, and $P$ is then weakly compact.
In particular, for any tight sequence of probability measures $\PP_n$, there is a probability measure $\PP$ and subsequence $\PP_{n_k}$ converging weakly to $\PP$ (see \citep{HeWangYan} Theorem 15.39.)
This allows us to find martingale measures with specified marginals as limits of sequences.

\begin{lemma}\label{lemma:lim of mgales on S}
Let $C\in\cp$ and $(\PP_n)_{n\in\nat}$ be a sequence of martingale measures on $(\D,\setsF)$ such that $\EP{\PP_n}{(X_t-x)_+}\rightarrow C(t,x)$.

Then, there exists a subsequence $\PP_{n_k}$ and a martingale measure $\PP$ on $(\D,\setsF)$ such that $\PP_{n_k}\rightarrow\PP$ in the sense of finite-dimensional distributions. Furthermore, $X$ is a martingale under $\PP$, continuous in probability and satisfies $\EP{\PP}{(X_t-x)_+}=C(t,x)$.
\end{lemma}
\begin{proof}
First, choose any $t\in\reals_+$ and $\epsilon>0$. For every $K>0$,
\begin{equation*}\begin{split}
{\PP_n}\left(|X_t|>K\right) &\le \EP{\PP_n}{(X_t-K+1)_++1+(X_t+K-1)_+-(X_t+K)_+}\\
&\rightarrow C(t,K-1)+1+C(t,1-K)-C(t,-K).
\end{split}\end{equation*}
As this can be made arbitrarily small by making $K$ large, we see that for every $\epsilon>0$ there exists a $K>0$ such that
$\PP_n\left(|X_t|>K\right)<\epsilon$
for every $n$. Letting $S=\{s_1,s_2,\ldots\}$ be a countable dense subset of $\reals_+$ and $\epsilon>0$, this shows that there exists a sequence $K_n>0$ such that
\begin{equation*}
\PP_n\left(|X_{s_n}|>K_n\right)<2^{-n}\epsilon.
\end{equation*}
Letting $A$ be the compact set of all $\omega\in\reals^S$ satisfying $|\omega(s_n)|\le K_n$,
\begin{equation*}
\PP^S_n(\reals^S\setminus A)\le\sum_{n=1}^\infty\PP_n(|X_{s_n}|>K_n)<\sum_{n=1}^\infty 2^{-n}\epsilon=\epsilon.
\end{equation*}
So the sequence $\PP^S_n$ is tight and, by passing to a subsequence if necessary, we may assume that it convergence weakly to a probability measure $\QQ$ on $(\reals^S,\setsF^S)$.

For every $t\in S$ and $x,y\in\reals$, weak convergence gives
\begin{equation*}\begin{split}
\EP{\QQ}{(X^S_t-x)_+-(X^S_t-y)_+}&=\lim_{n\rightarrow\infty}\EP{\PP_n}{(X_t-x)_+-(X_t-y)_+}\\
&=C(t,x)-C(t,y).
\end{split}\end{equation*}
Letting $y$ increase to infinity and using dominated convergence,
\begin{equation}\label{eq:pf:lim of mgales on S:1}
\EP{\QQ}{(X^S_t-x)_+}=C(t,x).
\end{equation}
If $s<t$ are in $S$, $Z:\reals^S\rightarrow\reals$ is $\setsF^S_s$-measurable, continuous and such that $ZX^S_s$ is bounded, and $0\le Z\le 1$ then,
\begin{equation*}\begin{split}
\EP{\QQ}{Z(X^S_s-x)_+} &=\lim_{n\rightarrow\infty}\EP{\PP^S_n}{Z(X^S_s-x)_+}
\le\lim_{n\rightarrow\infty}\EP{\PP^S_n}{Z(X^S_t-x)_+}\\
&\le\lim_{n\rightarrow\infty}\EP{\PP_n^S}{Z\left((X^S_t-x)_+-(X^S_t-y)_+\right)}+C(t,y)\\
&=\EP{\QQ}{Z\left((X^S_t-x)_+-(X^S_t-y)_+\right)}+C(t,y)\\
\end{split}\end{equation*}
Letting $y$ increase to infinity and using dominated convergence shows that $(X^S_s-x)_+$ is a $\QQ$-submartingale. So,
\begin{equation*}\begin{split}
\EP{\QQ}{ZX^S_s}&=\lim_{x\rightarrow-\infty}\EP{\QQ}{Z\left((X^S_s-x)_++x\right)}\\
&\le\lim_{x\rightarrow-\infty}\EP{\QQ}{Z\left((X^S_t-x)_++x\right)}=\EP{\QQ}{ZX^S_t}.
\end{split}\end{equation*}
Therefore, $X^S$ is a $\QQ$-submartingale. Furthermore, as $C\in\cp$, (\ref{eq:pf:lim of mgales on S:1}) shows that $\EP{\QQ}{X^S_t}$ is independent of $t$ and $X^S$ is a $\QQ$-martingale. This allows us to extend $X^S_t$ to all $t\in\reals_+$ using $X^S_t=\EP{\QQ}{X^S_u|\setsF^S_t}$ for any $u\ge t$ in $S$. We now show that $X^S$ is continuous in probability. As it is a martingale, it has almost-sure left and right limits $X^S_{t-}, X^S_{t+}$ for every $t\in\reals_+$ (for $t=0$ set $X^S_{t-}=X^S_0$). Assuming that $0\in S$, taking the difference of the right and left limits of equation (\ref{eq:pf:lim of mgales on S:1}) in $t$ and using the continuity of $C$ gives,
\begin{equation*}\begin{split}
0&=\EP{\QQ}{(X^S_{t+}-x)_+-(X^S_{t-}-x)_+}\\
&= \EP{\QQ}{1_{\{X^S_{t-}>x\}}(X^S_{t+}-X^S_{t-})+1_{\left\{X^S_{t-}>x>X^S_{t+}\textrm{ or }X^S_{t+}>x\ge X^S_{t-}\right\}}|X^S_{t+}-x|}\\
&=\EP{\QQ}{1_{\left\{X^S_{t-}>x>X^S_{t+}\textrm{ or }X^S_{t+}>x\ge X^S_{t-}\right\}}|X^S_{t+}-x|}.
\end{split}\end{equation*}
So, $\QQ(X^S_{t-}>x>X^S_{t+})=\QQ(X^S_{t+}>x\ge X^S_{t-})=0$ for every $x$, showing that $X^S_{t-}=X^S_{t+}$. As $X^S$ is a martingale and right-continuous in probability, it has a \cadlag\ version and, therefore, there is a measure $\PP$ on $(\D,\setsF)$ satisfying $\PP^S=\QQ$. Furthermore, $X$ is a martingale which is continuous in probability under $\PP$. Taking limits of $t\in S$ also shows that $\EP{\PP}{(X_t-x)_+}=C(t,x)$.

It only remains to show that $\PP_n\rightarrow\PP$ in the sense of finite dimensional distributions. We use proof by contradiction, so suppose that this is not the case. Then there would exist a random variable $Z$ of the form (\ref{eqn:rv Z for finite dim}) for a finite subset $F=\{t_1,\ldots,t_m\}$ of $\reals_+$ for which $\EP{\PP_n}{Z}$ does not converge to $\EP{\PP}{Z}$. Passing to a subsequence if necessary, we may suppose that
\begin{equation}\label{eq:pf:lim of mgales on S:2}
\EP{\PP_n}{Z}\ge\EP{\PP}{Z}+\epsilon
\end{equation}
for some $\epsilon>0$ and every $n$. Setting $S^\prime=S\cup F$ the above argument shows that, by passing to a further subsequence, there exists a measure $\PP^\prime$ on $(\D,\setsF)$ such that $\PP_n^{S^\prime}\rightarrow(\PP^\prime)^{S^\prime}$.
In particular, $(\PP^\prime)^{S}=\lim_{n\rightarrow\infty}\PP_n^S=\PP^S$ and, by right-continuity in $t$, it follows that $\PP=\PP^\prime$ and $\PP_n^{S^\prime}\rightarrow\PP^{S^\prime}$ contradicting (\ref{eq:pf:lim of mgales on S:2}).
\end{proof}

Combining Lemma \ref{lemma:lim of mgales on S} with the results of \citep{Lowther2} gives the following, which will be used to construct \acd\ martingale measures with specified marginals by taking limits of measures matching the marginals at finitely many times.

\begin{lemma}\label{lemma:acd conv to acd}
Let $\PP_n$ be a sequence of \acd\ martingale measures on $(\D,\setsF)$ and $C\in\cp$ be such that
$\EP{\PP_n}{(X_t-x)_+}\rightarrow C(t,x)$.
Then, there exists a subsequence $\PP_{n_k}$ converging in the sense of finite-dimensional distributions to an \acd\ martingale measure $\PP$ satisfying $\EP{\PP}{(X_t-x)_+}=C(t,x)$.
\end{lemma}
\begin{proof}
First, Lemma \ref{lemma:lim of mgales on S} says that there exists a subsequence $\PP_{n_k}$ converging in the sense of finite-dimensional distributions to a martingale measure $\PP$ satisfying (\ref{eqn:unique measure fitting marginals}) under which $X$ is continuous in probability.
However, Corollary 1.3 of \citep{Lowther2} states that under such a limit, $X$ is an almost-continuous diffusion.
\end{proof}

To complete the proof of the existence of the \acd\ martingale measure, it just needs to be shown that it is possible to fit the marginals arbitrarily closely and Lemma \ref{lemma:acd conv to acd} will provide us with the required limit. There are, however, many different ways in which we can go about this. For example, we could construct a diffusion as the solution of a stochastic differential equation to match smooth $C\in\cp$. Alternatively, finite state Markov chains could be used to approximate the marginals by finite distributions.
However, one way to exactly match the marginals at any finite set of times is to use a Skorokhod embedding to time change a Brownian motion, as we describe now. This uses the methods described in \citep{Hobson2}.

Fix any $C\in\cp$ and times $t_0<t_1$. Also, let $\mu_t$ be the associated one dimensional measures, satisfying $\int(y-x)_+\,d\mu_t(y)=C(t,x)$.
Then, define the distribution function $F_1(x)=C\pd{2}(t_1,x+)+1=\mu_{t_1}((-\infty,x])$.
For $u\in(0,1)$ set $\beta(u)=\inf\{x\in\reals:F_1(x)\ge u\}$, and let $g_u:[\beta(u),\infty)\rightarrow\reals$ be
\begin{equation*}
g_u(x)=C(t_1,\beta(u))+\left(x-\beta(u)\right)(u-1).
\end{equation*}
Then, $\alpha(u)\ge\beta(u)$ is chosen to satisfy $g_u(\alpha(u))=C(t_0,\alpha(u))$.
This uniquely defines $\alpha(u)$ when $C(t_1,\beta(u))>C(t_0,\beta(u))$ (see Figure \ref{fig:defining alpha beta}), otherwise we set $\alpha(u)=\beta(u)$. Also, set $F_{0,1}^*(x)=\inf\{u\in(0,1):\alpha(u)>x\}$. This is right-continuous and increasing from $0$ to $1$, so is another distribution function.

If $B$ is a Brownian motion with initial distribution $\mu_{t_0}$ and $S_t=\sup_{s\le t}B_s$ is its maximum process, then a stopping time $\tau$ can be defined by
\begin{equation}\label{eq:skorokhod st}
\tau=\inf\left\{t\in\reals_+:F_1(B_t)\le F_{0,1}^*(S_t)\right\}.
\end{equation}
Then $B^\tau$ is a uniformly integrable martingale and $B_\tau$ has law equal to $\mu_{t_1}$.
See \citep{Hobson2} for details (Proposition 2.2 and Corollary 2.1).

\begin{figure} 
\includegraphics{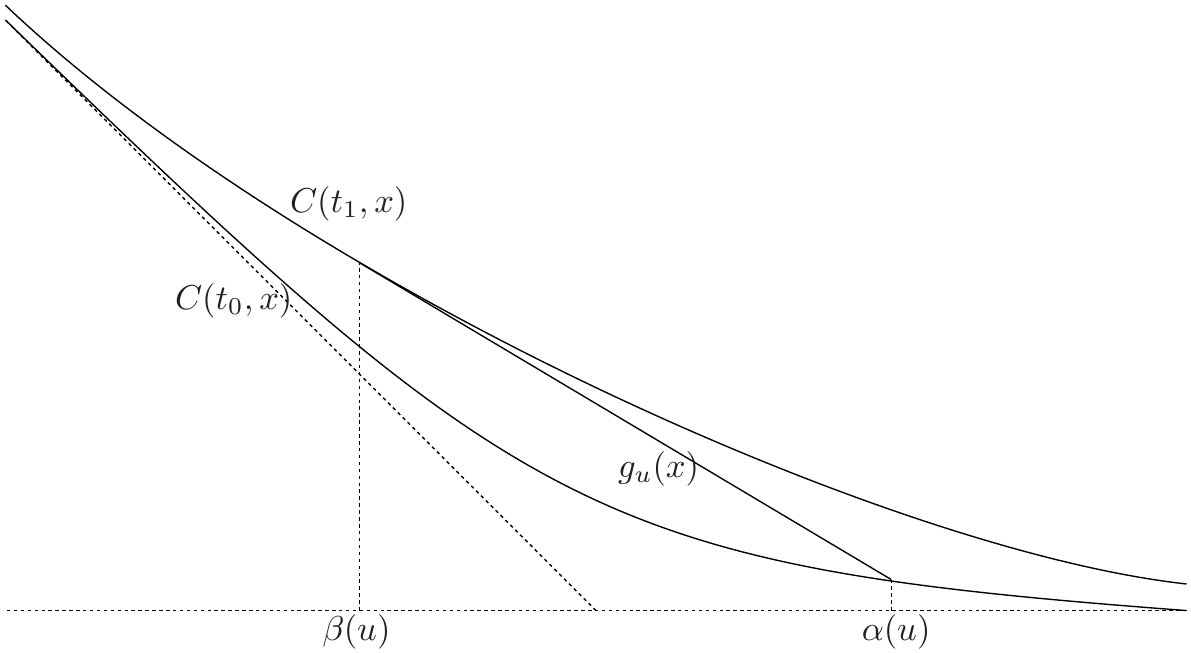} 
\caption{Fitting marginals to $C(t,x)$}\label{fig:defining alpha beta}
\end{figure} 

\begin{lemma}\label{lemma:2 marginals}
Let $C\in\cp$ and $t_0<t_1\in\reals_+$. Then, there exists an \acd\ martingale $X$ such that $\E{(X_t-x)_+}=C(t,x)$ for $t=t_0,t_1$.
\end{lemma}
\begin{proof}
Let $B$ be a Brownian motion with initial measure $\mu_{t_0}$ and $\tau$ be the stopping time (\ref{eq:skorokhod st}). Let $\theta:(t_0,t_1)\rightarrow\reals$ be any continuous function increasing from $-\infty$ to $\infty$. For example, $\theta(t)=(t_1-t)^{-1}-(t-t_0)^{-1}$. We also set $\theta(t)=-\infty$ for $t\le t_0$ and $\theta(t)=\infty$ for $t\ge t_1$. Defining the stopping times
\begin{equation*}
\tau_t=\inf\left\{s\in\reals_+:B_s\ge\theta(t)\right\},
\end{equation*}
the \acd\ martingale can then be constructed as $X_t=B_{\tau_t\wedge\tau}$.
The distributions of $X_{t_0}$ and $X_{t_1}$ are $\mu_{t_0}$ and $\mu_{t_1}$ (\citep{Hobson2} Proposition 2.2). As $B^\tau$ is uniformly integrable, $X$ will be a martingale. The paths of $X$ are very simple --- if $T$ is the first time at which $\tau_T\ge\tau$ then
\begin{equation*}
X_t=\left\{\begin{array}{ll}
\max(B_0,\theta(t)),&\textrm{if $t<T$},\\
B_\tau,&\textrm{if $t\ge T$.}
\end{array}\right.
\end{equation*}
So $X$ can only have a single jump at time $T$, at which $X_{T-}\ge X_T$.

It only remains to show that $X$ is an almost-continuous diffusion. Continuity in probability is easy. As $X_{t-}\ge X_t$ and $\E{X_{t-}}=\E{X_t}$ by the martingale property, it follows that $X_{t-}=X_{t}$ almost-surely.

The stopping times $\tau_t\wedge\tau$ are hitting times of the strong Markov process $(B_t,S_t)$, where $S_t=\sup_{s\le t}B_s$,  so the time changed process $(X_t,S_{\tau_t\wedge\tau})$ will also be strong Markov. However, $S_{\tau_t\wedge\tau}=X_{t}$ if $X_t\ge\theta(t)$ and $X$ is constant as soon as $X_t<\theta(t)$, so $X$ must be strong Markov.

We finally show that $X$ is almost-continuous, so choose a \cadlag\ process $Y$ independent and identically distributed as $X$. If $Y_s>X_s$, and $Y_t<X_t$ for $s<t$ then we can let $T$ be the first time at which $Y_T< X_T$, at which point we must have $Y_{T-}=X_{T-}=\theta(T)$. In fact, $T$ must be the first time at which $Y_{T-}=X_{T-}$, and so is previsible. So, the martingale property gives $\E{X_{T-}}=\E{X_T}$ and $\E{Y_{T-}}=\E{Y_T}$ and, as $X_{T-}\ge X_T$, $Y_{T-}\ge Y_T$, we have $Y_T=X_T$.
\end{proof}
This is easily extended to match the marginals at a finite set of times.
\begin{corollary}\label{cor:dense acd}
Let $C\in\cp$ and $A\subset\reals_+$ be finite. Then, there exists an \acd\ martingale measure $\PP$ on $(\D,\setsF)$ such that $\EP{\PP}{(X_t-x)_+}=C(t,x)$ for all $t\in A$.
\end{corollary}
\begin{proof}
If $A=\{t\}$ is a single time, then $\PP$ can be taken to be the measure under which $X_t$ is independent of $t$ with the required distribution. For $A=\{t_0<t_1<\cdots<t_n\}$ we use induction on $n$. Suppose that there is an \acd\ martingale measure $\PP_1$ matching the required marginals at times $t_0,t_1,\ldots,t_{n-1}$. By Lemma \ref{lemma:2 marginals} there is an \acd\ martingale measure $\PP_2$ matching the required marginals at times $t_{n-1},t_{n}$. Noting that $X_{t_{n-1}}$ has the same distribution under both $\PP_1$ and $\PP_2$, we can join these two measures together at time $t_{n-1}$ to get $\PP$, which is the unique measure on $(\D,\setsF)$ such that
\begin{equation*}
\EP{\PP}{AB}=\EP{\PP_1}{A\EP{\PP_2}{B|X_{t_{n-1}}}}=\EP{\PP_2}{\EP{\PP_1}{A|X_{t_{n-1}}}B}
\end{equation*}
for all bounded random variables $A,B$ where $A$ is $\setsF_{t_{n-1}}$-measurable and $B$ is $\sigma(X_t:t\ge t_{n-1})$-measurable. That $X$ is an \acd\ martingale under $\PP$ follows from the fact that it satisfies these properties over each of the intervals $[0,t_{n-1}]$ and $[t_{n-1},\infty)$.
\end{proof}

Finally for this section, Corollary \ref{cor:dense acd} is applied to construct the required \acd\ martingale measure.

\begin{lemma}\label{lemma:acd mart exists}
For every $C\in\cp$ there is an \acd\ martingale measure $\PP$ on $(\D,\setsF)$ satisfying $\EP{\PP}{(X_t-x)_+}=C(t,x)$.
\end{lemma}
\begin{proof}
By Corollary \ref{cor:dense acd}, there is a sequence of \acd\ martingale measures $\PP_n$ satisfying $\EP{\PP_n}{(X_{k/n}-x)_+}=C(k/n,x)$ for $k=0,1,\ldots,n$. Then $\EP{\PP_n}{(X_t-x)_+}\rightarrow C(t,x)$, so the existence of the \acd\ martingale measure $\PP$ follows from Lemma \ref{lemma:acd conv to acd}.
\end{proof}

\section{Uniqueness}
\label{sec:uniqueness}

Using a generalized form of the backward Kolmogorov equation, it was shown in \citep{Lowther4} that continuous and strong Markov martingales are uniquely determined by their marginal distributions. In this section we apply the arguments developed there to \acd\ martingales.
The idea is that if $X$ is a diffusion satisfying the stochastic differential equation (\ref{eq:SDE example}) and it is assumed that $C(t,x)=\E{(X_t-x)_+}$ is twice continuously differentiable, then the backward Kolmogorov equation 
\begin{equation*}
\frac{\partial f}{\partial t}+\frac{1}{2}\sigma(t,x)^2\frac{\partial^2 f}{\partial x^2}=0
\end{equation*}
can be combined with the forward equation (\ref{eq:fwd PDE example}) to obtain the following martingale condition for $f(t,X_t)$,
\begin{equation*}
\frac{\partial f}{\partial t}\frac{\partial^2 C}{\partial x^2}+\frac{\partial C}{\partial t}\frac{\partial^2 f}{\partial x^2}=0.
\end{equation*}
This applies to twice continuously differentiable functions $f$. Multiplying by a smooth $\theta(t,x)$ with compact support in $\ohalfplane$, integration by parts gives
\begin{multline}\label{eq:smooth mart cond}
\dblint\left(\frac{\partial f}{\partial t}\frac{\partial^2 C}{\partial x^2}+\frac{\partial C}{\partial t}\frac{\partial^2 f}{\partial x^2}\right)\theta\,dt\,dx\\
=\dblint\left(
\frac{\partial f}{\partial x}\frac{\partial C}{\partial x}\frac{\partial\theta}{\partial t}
-\frac{\partial\theta}{\partial x}\frac{\partial f}{\partial x}\frac{\partial C}{\partial t}
-\frac{\partial C}{\partial x}\frac{\partial\theta}{\partial x}\frac{\partial f}{\partial t}
\right)\,dt\,dx.
\end{multline}
This expression is defined for differentiable functions, and differentiability in $t$ can be relaxed by replacing terms such as $\int\cdot(\partial f/\partial t)\,dt$ by the Lebesgue-Stieltjes integral $\int\cdot\,d_tf$. The class of functions we consider is as follows.
\begin{definition}\label{def:nicefunc}
Denote by $\nicefunc$ the set of functions $f:\halfplane\rightarrow\reals$ such that
\begin{itemize}
\item $f(t,x)$ is Lipschitz continuous in $x$ and c\`adl\`ag in $t$,
\item for every $K_0<K_1\in\reals$ and $T\in\reals_+$ then
$
\dlimint{K_0}{K_1}{0}{T}\,|d_tf(t,x)|\,dx<\infty,
$
\item the left and right derivatives of $f(t,x)$ with respect to $x$ exist everywhere.
\end{itemize}
Also, let $\nicefuncK$ be the functions $f\in\nicefunc$ with compact support in $\ohalfplane$.
\end{definition}
In particular, if $X$ is a \cadlag\ martingale then $C(t,x)\equiv\E{(X_t-x)_+}$ will be convex in $x$ and \cadlag\ and increasing in $t$, so $C\in\nicefunc$.
Using $f^-(t,x)$ to denote the left limit of $f(t,x)$ in $t$, it was shown in \citep{Lowther4} (Lemma 2.2) that for $f,g\in\nicefunc$ the partial derivatives $\partial f(t,x)/\partial x$ and $\partial f^-(t,x)/\partial x$ exist almost everywhere with respect to the measure $\dblint\cdot\,|d_tg(t,x)|\,dx$. This enables the following definition to be made. Here, where we suppress the arguments of functions inside the integral signs they are understood to be $(t,x)$.
\begin{definition}\label{defn:mufc}
For every $f,g\in\nicefunc$ define the linear map $\mu_{[f,g]}:{\nicefuncK}\rightarrow\reals$
\begin{equation*}
\mu_{[f,g]}\left( \theta \right)=
\dblint f\pd{2}g\pd{2}\,d_t\theta\,dx
-\dblint \theta^-\pd{2} f^-\pd{2}\,d_tg\,dx-\dblint g^-\pd{2}\theta^-\pd{2}\,d_tf\,dx.
\end{equation*}
\end{definition}
Comparing with the right hand side of (\ref{eq:smooth mart cond}), we expect the martingale condition to be $\mu_{[f,C]}=0$, which is indeed the case for continuous processes. However, to include discontinuous processes jump terms need to be added, leading to the following definition.
\begin{definition}\label{def:mutfx}
Let $X$ be a \cadlag\ martingale. Then, for every $f\in\nicefunc$ define the linear map $\tilde\mu^X_f:{\nicefuncK}\rightarrow\reals$
\begin{equation*}
\tilde\mu^X_f\left( \theta \right)= \mu_{[f,C]}(\theta) + \E{\sum_{t>0}J^X_t(\theta,f)}
\end{equation*}
where $C\in\nicefunc$ is defined by $C(t,x)=\E{(X_t-x)_+}$ and
\begin{equation}\label{eqn:def:J}
J^X_t(\theta,f)\equiv\int_{X_{t-}}^{X_t}\left(f(t,x)-f(t,X_t)+(X_t-x)f^-\pd{2}(t,x)\right)\theta^-\pd{2}(t,x)\,dx.
\end{equation}
\end{definition}
The sum of $J^X_t(\theta,f)$ will be integrable (see \citep{Lowther4}), so $\tilde\mu^X_f$ is well defined. We have the following necessary martingale condition.
\begin{lemma}\label{lemma:nec mart cond}
Let $X$ be a \cadlag\ martingale and $f\in\nicefunc$. If $f(t,X_t)$ is a martingale then $\tilde\mu^X_f=0$.
\end{lemma}
See \citep{Lowther4}, Theorem 3.8.
To prove Theorem \ref{thm:unique measure fitting marginals}, the converse of this statement needs to be shown and, for almost-continuous processes, we need to demonstrate that the jump terms $J^X_t(\theta,f)$ can be eliminated to express the martingale condition solely in terms of $\mu_{[f,C]}$.

Fortunately, the jump terms will indeed drop out of the expression for $\tilde\mu^X_f$ as long as $f$ is chosen such that $f^-(t,x)$ is linear in $x$ across each of the connected components of the complement of the support of $X_t$.
To show this, we make use of the \emph{marginal support}, which was defined in \citep{Lowther1} to be the set constructed from the supports of the marginal distributions of $X$ as,
\begin{equation*}
\msupport{X}\equiv\left\{(t,x)\in\halfplane:x\in\support{X_t}\right\}.
\end{equation*}
We now prove the following.
\begin{lemma}\label{lemma:muf for ac}
Let $X$ be an \acd\ martingale and $f\in\nicefunc$ satisfy $f\pd{22}^-(t,x)=0$ for every $t>0$ and $x$ outside the support of $X_t$. Then, $\tilde\mu^X_f=\mu_{[f,C]}$.
\end{lemma}
\begin{proof}
It was shown in Corollary 4.8 of \citep{Lowther1} that for an almost-continuous process, the set
\begin{equation*}
S\equiv\left\{(t,x)\in\ohalfplane:X_{s-}<x<X_s\textrm{ or }X_s<x<X_{s-}\right\}
\end{equation*}
is almost surely disjoint from the marginal support of $X$. Then, the condition of the lemma gives $f\pd{22}^-=0$ on $S$, so
\begin{equation}\label{eq:pf:lemma:muf for ac:1}
f^-(t,x)-f^-(t,X_t)+f^-\pd{2}(t,x)(X_t-x)=0
\end{equation}
for all $(t,x)\in S$, almost surely. Finally, as there are only countably many times for which $f^-\not= f$ and $X$ is continuous in probability, $f(t,x)=f^-(t,x)$ whenever $X_t\not= X_{t-}$ and (\ref{eq:pf:lemma:muf for ac:1}) gives $J^X_t(\theta,f)=0$.
\end{proof}
The following result shows that it will be enough to only consider functions $f\in\nicefunc$ satisfying $f^-\pd{22}=0$ outside the marginal support of $X$.
\begin{lemma}\label{lemma:cond exp is conv fun}
Let $X$ be an \acd\ martingale, $t>0$ and $g:\reals\rightarrow\reals$ be convex with bounded derivative $k\le g^\prime\le K$. Then, there exists an $f:\halfplane\rightarrow\reals$ such that
\begin{equation}\label{eq:convex expectation}
f(s,X_s)=\E{g(X_t)|\setsF_s}
\end{equation}
for every $s<t$, where $f(s,x)$ is convex in $x$, right-continuous and decreasing in $s$, $k\le f\pd{2}\le K$ and $f^-\pd{22}=0$ outside the marginal support of $X$.
\end{lemma}
\begin{proof}
First, by theorems 1.5 and 1.6 of \citep{Lowther1}, $f(s,x)$ can be chosen to be convex in $x$ with derivative $k\le f\pd{2}\le K$ and satisfying equation (\ref{eq:convex expectation}) for all $s\le t$.
Theorem 1.7 of \citep{Lowther1} also says that $f(s,x)$ is continuous on the marginal support of $X$ over the range $s\le t$.
We can extend $f(s,x)$ across any bounded connected component of $\reals\setminus\support{X_s}$ by linear interpolation. Also, we can linearly extrapolate w.r.t. $x$ above the supremum of $\support{X_s}$ with gradient $K$ and below the infimum with gradient $k$.
This gives $f\pd{22}(s,x)=0$ outside the support of $X_s$. For simplicity, over $s> t$, fix $f(s,x)\le g(x)$ to be a linear function of $x$ independent of $s$.
Then, for $s<u\le t$, Jensen's inequality gives
\begin{equation*}
f(s,X_s)=\E{f(u,X_u)|\setsF_s}\ge f(t,X_s),
\end{equation*}
from which it follows that $f(s,x)\ge f(u,x)$ for every $x$ in the support of $X_s$. By linear interpolation and extrapolation of $f(s,x)$ outside the support of $X_s$, this inequality is satisfied for all $x$. So, $f(s,x)$ is decreasing in $s$. We now define $\tilde f(s,x)$ to be the limit of $f(u,x)$ as $u$ strictly decreases to $s$.
Then, $\tilde f$ is convex in $x$, right-continuous and decreasing in $s$, and has derivative satisfying $k\le\tilde f\pd{2}\le K$.
By continuity, $\tilde f(s,x)=f(s,x)$ whenever $s<t$ and $x$ is in the support of $X_s$, so $\tilde f$ also satisfies (\ref{eq:convex expectation}).

It only remains to show that $\tilde f^-\pd{22}(s,x)=0$ for $s\le t$ and $x$ outside the support of $X_s$. First, suppose that $x$ is in a bounded connected component $(a,b)$ of $\reals\setminus\support{X_s}$.
As $f(s,x)$ was chosen to be linearly interpolated outside the support of $X_s$,
\begin{equation*}
\tilde f^-(s,x)\ge  f(s,x) = (b-a)^{-1}\left( (b-x)f(s,a)+(x-a)f(s,b)\right).
\end{equation*}
The reverse inequality follows from the convexity of $\tilde f^-$ and $\tilde f^-=f$ at points $(s,a)$ and $(s,b)$ (by continuity). So, $\tilde f\pd{22}^-(s,x)=0$.

Now suppose that $x$ is in a connected component $(a,\infty)$ of $\reals\setminus\support{X_s}$.
As $X$ is a martingale with $X_s\le a$, we must have $X_u\le a$ for all $u\le s$. 
But $f(u,x)$ was chosen to be linearly extrapolated with gradient $K$ over $x>a$,
\begin{equation*}
\tilde f^-(s,x)=\lim_{u\uparrow\uparrow s}f(u,x) = \tilde f^-(s,a) + K(x-a).
\end{equation*}
so we again have $\tilde f\pd{22}^-(s,x)=0$.
Finally, by the same argument, $\tilde f\pd{22}^-(s,x)=0$ for $x$ in a connected component $(-\infty,b)$ of $\reals\setminus\support{X_s}$.
\end{proof}

For a general $f\in\nicefunc$, $f(t,X_t)$ need not be a semimartingale, so it will be necessary to consider the more general class of \emph{Dirichlet processes}. These processes were introduced by Follmer in \citep{Follmer2}, generalized to the non-continuous case in \citep{Stricker1} and \citep{Coquet}, then further extended to the form required here in \citep{Lowther3} and \citep{Lowther4}. So using the notation of \citep{Lowther4}, a process $V$ is said to have zero continuous quadratic variation if its quadratic variation exists and satisfies $[V]_t=\sum_{s\le t}\Delta V_s^2$. For brevity, we say that a process is a \zcqv\ process if it is \cadlag, adapted and has zero continuous quadratic variation. A Dirichlet process is then defined as the sum of a semimartingale and a \zcqv\ process. Quadratic covariations of Dirichlet processes are well defined, and for any semimartingale $X$ and Dirichlet process $Y$ the integral $\int X_-\,dY$ can be defined using integration by parts,
\begin{equation*}
\int_0^tX_{s-}\,dY_s\equiv X_tY_t-X_0Y_0-\int_0^t Y_{s-}\,dX_s-[X,Y]_t.
\end{equation*}

The following result then says that $\tilde\mu^X_f$ gives a measure of the drift of the Dirichlet process $f(t,X_t)$ for any $f\in\nicefunc$. Here, $\Dom(X)$ denotes the functions $\theta\in\nicefunc$ which decompose as $\theta(t,X_t)=M_t+V_t$ for a \cadlag\ martingale $M$ and finite variation process $V$ satisfying
\begin{equation*}
\E{\int_0^t1_{\{|X_{s-}|\le K\}}\,|dV_s|}<\infty
\end{equation*}
for all $t,K>0$.
\begin{lemma}\label{lemma:int theta dA}
Let $X$ be a \cadlag\ martingale and $f\in\nicefunc$. Then there is a unique decomposition
\begin{equation}\label{eq:fMA decomp}
f(t,X_t)=M_t+A_t
\end{equation}
for martingale $M$ and previsible \zcqv\ process $A$ with $A_0 = 0$. Furthermore, for any $\theta\in\nicefuncK$ with $\theta\in\Dom(X)$,
\begin{equation*}
\E{\sup_{t>0}\left|\int_0^t \theta^-(s,X_{s-})\,dA_s\right|}<\infty
\end{equation*}
and,
\begin{equation*}
\E{\int_0^\infty \theta^-(s,X_{s-})\,dA_s}=\tilde\mu^X_f(\theta).
\end{equation*}
\end{lemma}
For the proof of this, see \citep{Lowther4} lemmas 5.4 and 8.1.
The proof of the martingale condition $\tilde\mu^X_f=0$ for \acd\ martingales will make use of Lemma \ref{lemma:int theta dA} to show that the process $A$ in decomposition (\ref{eq:fMA decomp}) is zero. This will be done by showing that the time reversed conditional variation, $V^r_X(A)$ below, is zero. First, the reverse filtration $\setsG^X_\cdot$ is
\begin{equation*}
\setsG^X_t=\sigma\left(\{X_s:s\ge t\}\cup\{A\in\setsF:\Prob{A}=0\}\right),
\end{equation*}
and $V^r_X(A)$ is defined by
\begin{equation*}
V^r_X(A)=\sup\E{\sum_{k=1}^nZ_k(A_{t_k}-A_{t_{k-1}})}
\end{equation*}
where the supremum is taken over all sequences $0\le t_0\le t_1\le\cdots\le t_n$ in $\reals_+$ and all $\setsG^X_{t_k}$-measurable random variables $|Z_k|\le 1$.
The following result from \citep{Lowther4} (Lemma 7.2) will be required to bound the variation of $A$.
\begin{lemma}\label{lemma:bdd rev cond var gives int var}
Let $X$ be a c\`adl\`ag real valued process and $A$ be a \zcqv\ process such that $\sup_{t\ge 0}|A_t|$ is integrable and $V^r_X(A)<\infty$.

Suppose furthermore that $A_t-A_s$ is $\setsG^X_s$-measurable for all $t>s$ and that there exists a measurable $u:\halfplane\rightarrow\reals$ such that $\Delta A_t=u(t,X_t)$ for all $t> 0$.
Then, $A$ has integrable variation satisfying
\begin{equation*}
\E{\int\,|dA|}\le V^r_X(A).
\end{equation*}
\end{lemma}
In order to apply Lemma \ref{lemma:bdd rev cond var gives int var} it will be necessary to calculate $\Delta A$, which we do by making use of the quasi-left-continuity of $X$.
Recall that a \cadlag\ process $X$ is quasi-left-continuous if $X_{\tau-}=X_\tau$ for every previsible stopping time $\tau$.
\begin{lemma}\label{lemma:qlc}
Every \acd\ martingale is quasi-left-continuous.
\end{lemma}
\begin{proof}
Let $X$ be an \acd\ martingale, $\alpha>0$, and $g:\reals\rightarrow\reals$ be bounded and Lipschitz continuous with coefficient $K$. Then, there exists a bounded $f_\alpha:\ohalfplane\rightarrow\reals$ such that $f_\alpha(t,x)$ is Lipschitz continuous in $x$ with coefficient $K$ and,
\begin{equation}\label{eq:pf:qlc:1}
f_\alpha(t,X_t)=\E{g(X_{t+\alpha})|\setsF_t}
\end{equation}
for every $t> 0$ (see \citep{Lowther1} Theorem 1.5 or \citep{Lowther2} Lemma 4.3). By linearity, this extends to all stopping times taking only finitely many values in $(0,\infty)$.

If $u<t$ and $h:\reals\rightarrow\reals$ is continuous and bounded, then the continuity of $X$ at $t$ together with bounded convergence gives
\begin{equation*}\begin{split}
\E{h(X_u)f_\alpha(t,X_t)}&=\E{h(X_u)g(X_{t+\alpha})}=\lim_{s\rightarrow t}\E{h(X_u)g(X_{s+\alpha})}\\
&=\lim_{s\rightarrow t}\E{h(X_u)f_\alpha(s,X_s)}
=\lim_{s\rightarrow t}\E{h(X_u)f_\alpha(s,X_t)}.
\end{split}\end{equation*}
If we let $u$ increase to $t$ then $\E{|h(X_u)-h(X_t)|}$ will go to $0$, and we see that the above equality also holds for $u=t$. As $f_\alpha(t,x)$ is uniformly continuous in $x$, this shows that $f_\alpha(t,x)$ is continuous in $t$ for every $x$ in the support of $X_t$. So, $f_\alpha$ is continuous on the marginal support of $X$. By lemmas 4.3 and 4.4 of \citep{Lowther1}, the continuity in probability of $X$ implies that the paths of $(t,X_t)$ and $(t,X_{t-})$ lie in the marginal support of $X$ at all times. So, by taking limits of stopping times which take finitely many values in $(0,\infty)$, equation (\ref{eq:pf:qlc:1}) extends to all stopping times. Similarly, taking increasing limits of stopping times gives
\begin{equation*}
f_\alpha(\tau,X_{\tau-})=\E{g(X_{\tau+\alpha-})|\setsF_{\tau-}}
\end{equation*}
for all previsible stopping times $\tau>0$.
We now note that $f_\alpha(t,X_t)\rightarrow g(X_t)$ in probability as $\alpha\rightarrow 0$. So, by uniform continuity in $x$, $f_\alpha(t,x)\rightarrow g(x)$ on the marginal support of $X$, giving
\begin{equation*}
g(X_{\tau-})=\lim_{\alpha\rightarrow 0}f_\alpha(\tau,X_{\tau-})=\lim_{\alpha\rightarrow 0}\E{g(X_{\tau+\alpha-})|\setsF_{\tau-}}=\E{g(X_\tau)|\setsF_{\tau-}}.
\end{equation*}
So, if $h(x)$ is any bounded measurable function and $u(x,y)=h(x)g(y)$, we can multiply by $h(X_{\tau-})$ and take expectations,
\begin{equation*}
\E{u(X_{\tau-},X_{\tau-})}=\E{u(X_{\tau-},X_\tau)}.
\end{equation*}
Finally, by the monotone class lemma, this extends to all bounded measurable functions $u:\reals^2\rightarrow\reals$ and taking $u(x,y)=1_{\{x=y\}}$ gives $\Prob{X_{\tau-}=X_\tau}=1$.
\end{proof}
Finally, we prove the martingale condition $\mu_{[f,C]}=0$.
\begin{theorem}\label{thm:mart cond}
Let $X$ be an \acd\ martingale and $f\in\nicefunc$ satisfy $f\pd{22}^-(t,x)=0$ for $t>0$ and $x$ outside the support of $X_t$. Then, $f(t,X_t)$ is a martingale if and only if $\mu_{[f,C]}=0$.
\end{theorem}
\begin{proof}
First, Lemma \ref{lemma:muf for ac} gives $\tilde\mu^X_f=\mu_{[f,C]}$.
Then, if $f(t,X_t)$ is a martingale Lemma \ref{lemma:nec mart cond} gives $\mu_{[f,C]}=0$. Conversely, suppose that $\mu_{[f,C]}=0$ and let $f(t,X_t)=M_t+A_t$ be decomposition (\ref{eq:fMA decomp}). Also, choose any twice continuously differentiable $\theta:\ohalfplane\rightarrow\reals$ with compact support and set $B=\int\theta(s,X_{s-})\,dA_s$, which is a previsible \zcqv\ process (see \citep{Lowther4} Lemma 5.5). We will show that the necessary conditions to apply Lemma \ref{lemma:bdd rev cond var gives int var} to $B$ are satisfied.

First, $\E{\Delta M\tau|\setsF_{\tau-}}=0$ for every previsible stopping time $\tau$, and the quasi-left-continuity of $X$ gives
\begin{equation*}
\Delta A_\tau=\E{f(\tau,X_\tau)-f^-(\tau,X_{\tau-})|\setsF_{\tau-}}=f(\tau,X_{\tau-})-f^-(\tau,X_{\tau-}).
\end{equation*}
By previsible section (e.g., \citep{HeWangYan} Corollary 4.11), it follows that this equality holds simultaneously at all times. In particular
\begin{equation}\label{eq:pf:mart cond:2}
\Delta B_t=\theta(t,X_t)\left( f(t,X_t)-f^-(t,X_t)\right)
\end{equation}
for all times at which $\Delta X_t=0$. As $f=f^-$ at all but countably many times, and $X$ is continuous in probability, the right hand side of (\ref{eq:pf:mart cond:2}) equals $0$ when $\Delta X\not=0$. Also, as $B$ is previsible and $X$ is quasi-left-continuous, the left hand side is also zero when $\Delta X_t\not=0$. So, (\ref{eq:pf:mart cond:2}) holds for all times, and $\Delta B_t$ is a function of $(t,X_t)$, as required by Lemma \ref{lemma:bdd rev cond var gives int var}.

We now show that $V^r_X(B)=0$. To do this, first choose any $T>0$ and let $g:\reals\rightarrow\reals$ be bounded and a difference of convex Lipschitz continuous functions. Lemma \ref{lemma:cond exp is conv fun} says that there is an $h\in\nicefunc$ such that $h\pd{22}^-(t,x)=0$ for $x$ outside the support of $X_t$ and,
\begin{equation*}
h(t,X_t)=\E{g(X_T)|\setsF_t}
\end{equation*}
for $t<T$. Then, integration by parts gives
\begin{equation}\label{eq:pf:mart cond:3}
h(t,X_t)(B_t-B_s) = \int_s^t h^-(u,X_{u-})\theta(u,X_{u-})\,dA+\int_s^tA_{u}\,dh(u,X_{u-})
\end{equation}
(see \citep{Lowther4} equation 5.5 and Lemma 5.6). As $h(t,X_t)$ is a martingale over $t<T$ it is clear that $h\in\Dom(X)$. Also, $\theta$ is twice continuously differentiable and Ito's formula shows that it is in $\Dom(X)$ (see \citep{Lowther4} Lemma 4.3). So, by Lemma 8.2 of \citep{Lowther4}, the product $\theta f$ is also in $\Dom(X)$. So, we can take expectations of (\ref{eq:pf:mart cond:3}) to get
\begin{equation}\label{eq:pf:mart cond:4}
\E{g(X_T)(B_t-B_s)}=\E{h(t,X_t)(B_t-B_s)}=\mu_{[f,C]}(1_{[s,t)}h\theta)=0
\end{equation}
for $s<t<T$. Here, we have used that fact that the final term on the right hand side of (\ref{eq:pf:mart cond:3}) is a local martingale and applied Lemma \ref{lemma:int theta dA}.
Taking limits as $T$ decreases to $t$, (\ref{eq:pf:mart cond:4}) holds for $t=T$.
Then, the monotone class lemma shows that (\ref{eq:pf:mart cond:4}) holds for all bounded and measurable functions $g$. So, if $Z$ is a bounded $\setsG^X_t$-measurable random variable, the Markov property for $X$ gives
\begin{equation*}
\E{Z_t(B_t-B_s)}=\E{\E{Z|X_t}(B_t-B_s)}=0,
\end{equation*}
and, $V^r_X(B)=0$.

Also, $A_t-A_s$ (and therefore $B_t-B_s$) is $\setsG^X_s$-measurable for all $s<t$ (\citep{Lowther4} Corollary 7.5). So Lemma \ref{lemma:bdd rev cond var gives int var} shows that $B$ has zero variation and is therefore constant. Then, as $\theta$ is arbitrary, $A$ must also be constant and $f(t,X_t)=M_t$ is a martingale.
\end{proof}

We end this section with the proofs of  theorems \ref{thm:unique measure fitting marginals} and \ref{thm:marginals to martingales is cts}.

\begin{proof}[Proof of Theorem \ref{thm:unique measure fitting marginals}]
Lemma \ref{lemma:acd mart exists} shows existence of the \acd\ martingale measure.
To prove uniqueness, suppose that we have two such measures $\PP$ and $\QQ$.
Choose any $t>0$ and convex and Lipschitz continuous $g:\reals\rightarrow\reals$. By Lemma \ref{lemma:cond exp is conv fun} there is an $f\in\nicefunc$ such that $f\pd{22}^-=0$ outside the marginal support of $X$ and satisfying
\begin{equation*}
f(s,X_s)=\EP{\PP}{g(X_t)|\setsF_s}
\end{equation*}
for all $s<t$. Letting $s$ increase to $t$ gives $f^-(t,X_t)=g(X_t)$ $\PP$-a.s..
So, we may set $f(s,x)=f^-(t,x)$ for all $s\ge t$, and $f(s,X^t_s)$ will be a $\PP$-martingale.

However, it is clear that the stopped process $X^t$ is also an \acd\ martingale under both $\PP$ and $\QQ$. So, setting $\tilde C(s,x)=C(s\wedge t,x)$ we can apply Theorem \ref{thm:mart cond} to get $\mu_{[f,\tilde C]}=0$.
Another application of Theorem \ref{thm:mart cond} shows that $f(s,X^t_s)$ is also a $\QQ$-martingale and,
\begin{equation*}
f(s,X_s)=\EP{\QQ}{f(t,X_t)|\setsF_s}=\EP{\QQ}{g(X_t)|\setsF_s}.
\end{equation*}
So $X$ has the same pairwise distributions under both $\PP$ and $\QQ$. As they are Markov measures with common initial distribution, this gives $\PP=\QQ$.
\end{proof}

\begin{proof}[Proof of Theorem \ref{thm:marginals to martingales is cts}]
Choose any $C\in\cp$ and sequence $C_n\in\cp$ such that $C_n\rightarrow C$.
We use proof by contradiction to show that $\PP_{C_n}\rightarrow\PP_C$, so suppose that this is false. Then there would exist an $\epsilon>0$ and a random variable $Z$ of the form (\ref{eqn:rv Z for finite dim}) such that
\begin{equation}\label{eq:pf:marginals to martingales is cts:1}
\EP{\PP_{C_n}}{Z}\ge\EP{\PP_C}{Z}+\epsilon
\end{equation}
infinitely often. By passing to a subsequence if necessary, we may suppose that this inequality holds for every $n$.
Then, Lemma \ref{lemma:acd conv to acd} says that by passing to a further subsequence we have $\PP_{C_n}\rightarrow\PP$ for an \acd\ martingale measure $\PP$ satisfying (\ref{eqn:unique measure fitting marginals}). So, the uniqueness part of Theorem \ref{thm:unique measure fitting marginals} gives $\PP=\PP_{C}$, and $\PP_{C_n}\rightarrow\PP_C$ contradicting (\ref{eq:pf:marginals to martingales is cts:1}).
\end{proof}

\section{Extremal Marginals}
\label{sec:extremal}

Note that the space $\cp$ is a convex subset of the space of all real-valued functions on $\halfplane$. For any convex subset of a vector space there is a concept of extremal points --- they are the points which cannot be expressed as a convex combination of other elements in the set.
More precisely, for a convex subset $S$ of a vector space $V$, an element $x\in S$ is said to be extremal if given any $y,z\in S$ and any $\lambda\in(0,1)$ such that $x = \lambda y + (1-\lambda) z$ then $y=z=x$.
We shall call an element of $\cp$ extremal if it is extremal among the convex set of all other elements of $\cp$ with the same values at time $0$.

\begin{definition}
A $C\in\cp$ is \emph{extremal} if given any $C_1,C_2\in\cp$ and $\lambda\in(0,1)$ such that $C_1(0,x)=C_2(0,x)=C(0,x)$ $(\forall  x\in\reals)$ and $C=\lambda C_1 + ( 1 - \lambda) C_2$ then $C_1=C_2=C$.
\end{definition}

It can then be shown that there is a unique martingale measure fitting the marginals given by an extremal element of $\cp$.

\begin{lemma}\label{lemma:extremal give unique martingale}
Let $C\in\cp$ be extremal. Then, there exists a unique martingale measure $\PP$ on $(\D,\setsF)$ satisfying $\EP{\PP}{(X_t-x)_+}=C(t,x)$.
\end{lemma}
\begin{proof}
As Theorem \ref{thm:unique measure fitting marginals} says that such a martingale measure exists, we only need to prove uniqueness. Let us start by showing that any such measure $\PP$ is Markov. Given any $T\ge 0$ and $\setsF_T$-measurable random variable $Z$ with $0\le Z\le 1$, define $C_1,C_2$ by $C_1(t,x)=C_2(t,x)=C(t,x)$ for $t<T$ and,
\begin{align}
&C_1(t,x) = \EP{\PP}{Z\EP{\PP}{(X_t-x)_+|X_T}}+\EP{\PP}{(1-Z)(X_t-x)_+},\label{eq:pf:extremal give unique martingale:1}\\
&C_2(t,x) = \EP{\PP}{(1-Z)\EP{\PP}{(X_t-x)_+|X_T}}+\EP{\PP}{Z(X_t-x)_+}\nonumber
\end{align}
for $t\ge T$. It is easily checked that $C_1,C_2$ are in $\cp$, $C_1(0,x)=C_2(0,x)=C(0,x)$ and $(C_1+C_2)/2=C$. As $C$ is extremal, this implies $C_1=C$.
For $t\ge T$, putting $C_1(t,x)=\EP{\PP}{(X_t-x)_+}$ into (\ref{eq:pf:extremal give unique martingale:1}) gives
\begin{equation*}
\EP{\PP}{Z(X_t-x)_+}=\EP{\PP}{Z\EP{\PP}{(X_t-x)_+|X_T}},
\end{equation*}
so $\EP{\PP}{(X_t-x)_+|\setsF_T}=\EP{\PP}{(X_t-x)_+|X_T}$, and $X$ is indeed Markov.

Let us now suppose that $\PP$, $\QQ$ are two such martingale measures. Choose any $T\ge 0$ and $X_T$-measurable random variable $Z$ with $0\le Z\le 1$, and define $C_1,C_2$ by $C_1(t,x)=C_2(t,x)=C(t,x)$ for $t<T$ and,
\begin{align}
&C_1(t,x)=\EP{\PP}{Z(X_t-x)_+}+\EP{\QQ}{(1-Z)(X_t-x)_+},\label{eq:pf:extremal give unique martingale:2}\\
&C_2(t,x)=\EP{\QQ}{Z(X_t-x)_+}+\EP{\PP}{(1-Z)(X_t-x)_+}\nonumber
\end{align}
for $t\ge T$. Again, it is easily checked that $C_1,C_2$ are in $\cp$ with $C_1(0,x)=C_2(0,x)=C(0,x)$ and $(C_1+C_2)/2=C$. As $C$ is extremal this implies $C_1=C$. For $t\ge T$, putting $C_1(t,x)=\EP{\QQ}{(X_t-x)_+}$ into (\ref{eq:pf:extremal give unique martingale:2}) gives
\begin{equation*}
\EP{\PP}{Z(X_t-x)_+}=\EP{\QQ}{Z(X_t-x)_+}
\end{equation*}
for all $t\ge T$. Therefore, $\PP$ and $\QQ$ are Markov measures for $X$ with the same pairwise and initial distributions, so $\PP=\QQ$.
\end{proof}

Theorem \ref{thm:acd martingale method is unique cts} will follow once it is shown that the extremal elements are dense in $\cp$.
For $C\in\cp$ we will construct extremal elements of $\cp$ which match $C$ any given increasing sequence of times.
We start by showing that there is an extremal element matching $C$ at two times. The method used here corresponds to the \acd\ martingales constructed in Section \ref{sec:existence} (Lemma \ref{lemma:2 marginals}), which was based on the Skorokhod embedding described in \citep{Hobson2}, so let us recall some of the definitions from Section \ref{sec:existence}.

Given a $C\in\cp$ let $\mu_t$ be the corresponding marginal distributions satisfying $\int(y-x)_+\,d\mu_t(y)=C(t,x)$. Fixing $t_0<t_1$, define the distribution function $F_1(x)= C\pd{2}(t_1,x+)+1=\mu_{t_1}((-\infty,x])$ and for $u\in(0,1)$ set $\beta(u)=\inf\{x\in\reals:F_1(x)\ge u\}$. For $x\ge\beta(u)$ set
\begin{equation*}
g_u(x)=C(t_1,\beta(u))+\left(x-\beta(u)\right)(u-1)
\end{equation*}
and define $\alpha(u)\ge\beta(u)$ by $g_u(\alpha(u))=C(t_0,\alpha(u))$. This uniquely defines $\alpha(u)$ whenever $C(t_1,\beta(u))>C(t_0,\beta(u))$ (see Figure \ref{fig:defining alpha beta}), otherwise we take $\beta(u)=\alpha(u)$.

We define $\tilde C:\halfplane\rightarrow\reals$ by setting $\tilde C(t,x)$ equal to $C(t_0,x)$ for $t\le t_0$, $C(t_1,x)$ for $t\ge t_1$ and,
\begin{equation}\label{eqn:def of extremal C 2}
\tilde C(t,x)=\left\{
\begin{array}{ll}
C(t_1,x),&\textrm{if }x\le \beta(u),\\
C(t_0,x),&\textrm{if }x\ge \alpha(u),\\
g_u(x),&\textrm{if }\beta(u)<x<\alpha(u)
\end{array}
\right.
\end{equation}
for $t_0<t<t_1$, where $u$ is set to $u(t)\equiv (t-t_0)/(t_1-t_0)$.
We show that this does indeed give an extremal element of $\cp$.

\begin{lemma}\label{lemma:C is extremal}
Suppose that $C\in\cp$ and $t_0<t_1\in\reals_+$. Then $\tilde C$ defined by (\ref{eqn:def of extremal C 2}) is an extremal element of $\cp$ such that $\tilde C(t,x)$ equals $C(t_0,x)$ for $t\le t_0$ and $C(t_1,x)$ for $t\ge t_1$.
\end{lemma}
\begin{proof}
For $t_0<t<t_1$, $\tilde C(t,x)$ is a convex function of $x$ lying between $C(t_0,x)$ and $C(t_1,x)$ (see Figure \ref{fig:defining alpha beta}). To show that $\tilde C\in\cp$ it just needs to be shown that it is continuous and increasing in $t$.
So, pick any $s<t\in(t_0,t_1)$. Then, $u(s)\le u(t)$ and $\beta(u(s))\le \beta(u(t))$. For any $x\le \beta(u(t))$ this gives
\begin{equation*}
\tilde C(t,x)=C(t_1,x)\ge \tilde C(s,x).
\end{equation*}
Also, if $x> \beta(u(t))$ then
\begin{equation*}
\tilde C(t,x) = \max(g_{u(t)}(x),C(t_0,x)) \ge \max(g_{u(s)}(x),C(t_0,x)) = \tilde C(s,x).
\end{equation*}
So, $\tilde C(t,x)$ is increasing in $t$.

Now choose any $x,y\in\reals$ with $C(t_0,x)<y<C(t_1,x)$, and choose $b<x$ to minimize $u=(C(t_1,b)-y)/(x-b)$. To see that this exists, note that choosing $b$ small enough so that $C(t_1,b)-C(t_0,b)<y-C(t_0,x)$ gives
\begin{equation*}
(C(t_1,b)-y)/(x-b)<(C(t_0,b)-C(t_0,x))/(x-b)\le 1
\end{equation*}
but the limit as $b\rightarrow-\infty$ is $1$. So, by continuity, it must have a minimum.

Choosing $t$ such that $u(t)=1-(C(t_1,x)-y)/(b-x)$ gives $C\pd{2}(t_1,b+)\ge u-1\ge C\pd{2}(t_1,b-)$, and it follows that $\tilde C(t,x)=C(t_1,b)+(x-b)(1-u)=y$.
Therefore, $t\mapsto \tilde C(t,x)$ maps the interval $[t_0,t_1]$ onto $[C(t_0,x),C(t_1,x)]$ and must be continuous.

It only remains to show that $\tilde C$ is extremal, so suppose that
$\tilde C=\lambda C_1 +(1-\lambda)C_2$ for $C_1,C_2\in\cp$, $\lambda\in(0,1)$, and
\begin{equation*}
C_1(0,x)=C_2(0,x)=\tilde C(0,x)=C(t_0,x).
\end{equation*}
In order to show that $C_1=C_2=\tilde C$ we shall make repeated use of the simple fact that if a non-trivial convex combination of two increasing functions is constant, then those functions must also be constant.
In particular, $\tilde C(t,x)=\tilde C(t_1,x)$ for all $t\ge t_1$ so we must also have $C_i(t,x)=C_i(t_1,x)$ for $i=1,2$ and $t\ge t_1$. Similarly, $C_i(t,x)=C_i(0,x)=C(t_0,x)$ for $t\le t_0$.

Now choose any $t\in(t_0,t_1)$ and set $\beta=\beta(u(t))$, $\alpha=\alpha(u(t))$. Then for $x\le\beta$, the definition of $\tilde C$ gives $\tilde C(t,x)=\tilde C(t_1,x)$, so it is also true that $C_i(t,x)=C_i(t_1,x)$ for $i=1,2$.

Also, if $x\ge \alpha$ then $\tilde C(t,x)=\tilde C(t_0,x)$. Therefore $C_i(t,x)=C_i(t_0,x)=C(t_0,x)$.
Furthermore, for all $x$ in $(\beta,\alpha)$
\begin{equation*}
\lambda (C_1)\pd{2}(t,x+)+(1-\lambda)(C_2)\pd{2}(t,x+)=\tilde C\pd{2}(t,x+)=u(t)-1.
\end{equation*}
As $(C_i)\pd{2}(t,x+)$ are increasing functions of $x$, this shows that they are constant as $x$ runs through this interval. Therefore, $C_i(t,x)$ are linear functions of $x$ over $(\beta,\alpha)$.
So, we have shown that
\begin{equation*}
C_i(t,x)=\left\{
\begin{array}{ll}
C_i(t_1,x),&\textrm{if }x\le \beta,\\
C(t_0,x),&\textrm{if }x\ge \alpha,\\
\left((\alpha-x)C_i(t_1,\beta)+(x-\beta)C(t_0,\alpha)\right)/(\alpha-\beta),&\textrm{if }\beta<x<\alpha.
\end{array}
\right.
\end{equation*}
Now suppose that there exists an $x\in\reals$ such that $C_1(t_1,x)<C(t_1,x)$. Choose $t$ such that $u(t)=C\pd{2}(t_1,x+)+1$ and set $\beta=\beta(u(t))$, $\alpha=\alpha(u(t))$. It follows that $\beta\le x<\alpha$, so the fact that $C_1(t,x)$ and $C(t,x)$ are convex in $x$ and increasing in $t$ gives
\begin{equation*}\begin{split}
(C_1)\pd{2}(t_1,x+) &\ge \left( C(t_0,\alpha)-C_1(t_1,\beta)\right)/(\alpha-\beta)\\
&\ge \left(C(t_0,\alpha)-C(t_1,\beta)\right)/(\alpha-\beta)\\
&\ge C\pd{2}(t_1,x-).
\end{split}\end{equation*}
Taking the right hand limits in $x$ gives $(C_1)\pd{2}(t_1,x+)\ge C\pd{2}(t_1,x+)$.

Similarly, if $C_1(t_1,x)>C(t_1,x)$ then $C_2(t_1,x)<C(t_1,x)$ and the above argument gives $(C_2)\pd{2}(t_1,x+)\ge C\pd{2}(t_1,x+)$. So $(C_1)\pd{2}(t_1,x+)\le C\pd{2}(t_1,x+)$.

This shows that the function $g(x)=\left(C(t_1,x)-C_1(t_1,x)\right)^2$ has a non-positive derivative everywhere and must be decreasing, so the limit $g(x)\rightarrow 0$ as $x\rightarrow \pm\infty$ implies that $g$ is identically $0$. So $C_1(t_1,x)=C(t_1,x)$, from which it follows that $C_1=C$.
\end{proof}

The previous lemma allows us to construct an extremal element of $\cp$ matching $C$ at an increasing sequence of times.

\begin{corollary}\label{cor:extremal exists at sequence}
Suppose that $C\in\cp$ and $t_0<t_1<\cdots\uparrow\infty$ are in $\reals_+$.
Then there is an extremal $\tilde C\in\cp$ such that $\tilde C(t_k,x)=C(t_k,x)$ for each $k$.
\end{corollary}
\begin{proof}
Without loss of generality, we assume that $t_0=0$. By Lemma \ref{lemma:C is extremal} there exists extremal $\tilde C_k\in\cp$ such that $\tilde C_k(t,x)$ equals $C(t_{k-1},x)$ for $t\le t_{k-1}$ and equals $C(t_k,x)$ for $t\ge t_k$ ($k=1,2,\ldots$). Define $\tilde C$ by $\tilde C(t,x) =\tilde C_k(t,x)$ for $t_{k-1}\le t< t_k$.

It is clear that $\tilde C\in\cp$. It just needs to be shown that it is extremal. So suppose that
\begin{equation}\label{eq:pf:extremal exists at sequence:1}
\tilde C=\lambda C_1+(1-\lambda)C_2
\end{equation}
for $\lambda\in(0,1)$, $C_1,C_2\in\cp$, and $C_1(0,x)=C_2(0,x)=C(0,x)$.
We use induction on $k$ to show that $C_1(t,x)=C_2(t,x)=C(t,x)$ for $t\le t_k$. For $k=0$ this is a required condition, so suppose that $k\ge 1$ and that this holds for all $t\le t_{k-1}$.
If $\theta(t)\equiv(t\wedge t_k)\vee t_{k-1}$ then $\tilde C(\theta(t),x)=\tilde C_k(t,x)$ is extremal, and (\ref{eq:pf:extremal exists at sequence:1}) gives $\tilde C_k(t,x)=C_1(t,x)=C_2(t,x)$ for $t_{k-1}\le t\le t_k$.

So, by induction, $\tilde C=C_1=C_2$.
\end{proof}

Finally we complete the proof of Theorem \ref{thm:acd martingale method is unique cts}, showing that $C\mapsto\PP_C$ given by Theorem \ref{thm:unique measure fitting marginals} is the unique continuous map to the martingale measures and matching all possible sets of marginals.

\begin{proof}[Proof of Theorem \ref{thm:acd martingale method is unique cts}]
Choose any $C\in S$ and $Z$ be a random variable of the form (\ref{eqn:rv Z for finite dim}).
We just need to show that $\EP{\QQ_C}{Z}=\EP{\PP_C}{Z}$.

First, Corollary \ref{cor:extremal exists at sequence} gives a sequence $(C_n)_{n\in\nat}$ of extremal elements of $\cp$ such that $C_n(k/n,x)=C(k/n,x)$ for all $k\in\nat$ and, consequently, $C_n\rightarrow C$.

Then, as $S$ is dense in $\cp$, there is a sequence $(C_{n,m})_{m\in\nat}$ such that $d(C_{n,m},C_n)\rightarrow 0$ as $m\rightarrow\infty$, where $d$ is the metric on $\cp$ given by (\ref{eqn:cp metric}).

Now, fixing any $n\in\nat$, Lemma \ref{lemma:lim of mgales on S} says that, by passing to a subsequence if necessary, there exists a martingale measure $\PP$ satisfying equation (\ref{eqn:unique measure fitting marginals}) and such that $\QQ_{C_{n,m}}\rightarrow\PP$ in the sense of finite-dimensional distributions as $m\rightarrow\infty$.
However $C_n$ is extremal so, by Lemma \ref{lemma:extremal give unique martingale}, $\PP=\PP_{C_n}$.
Therefore, for every $n$ we can choose an $m_n\in\nat$ such that
\begin{equation*}
d(C_{n,m_n},C_n)<2^{-n},\ \left|\EP{\QQ_{C_{n,m_n}}}{Z}-\EP{\PP_{C_n}}{Z}\right|<2^{-n}.
\end{equation*}
In particular, $d(C,C_{n,m_n})\le d(C,C_n) + 2^{-n}$ so the continuity of $C\mapsto \QQ_C$ shows that $\QQ_{C_{n,m_n}}$ tends to $\QQ_C$ in the sense of finite dimensional distributions as $n$ goes to infinity. Similarly, $\PP_{C_n}$ tends to $\PP_C$ giving,
\begin{equation*}
\left|\EP{\PP_C}{Z}-\EP{\QQ_C}{Z}\right|
= \lim_{n\rightarrow\infty}\left|\EP{\PP_{C_n}}{Z}-\EP{\QQ_{C_{n.m_n}}}{Z}\right|\le\lim_{n\rightarrow\infty}2^{-n}=0.\qedhere
\end{equation*}
\end{proof}

\bibliography{martingales.bbl}

\begin{thebibliography}{10}

\bibitem{Andersen}
Leif Andersen and Jasper Andreasen.
\newblock Jump-diffusion processes: Volatility smile fitting and numerical
  methods for option pricing.
\newblock {\em Review of Derivatives Research}, 4(3):231--262, 2000.

\bibitem{ImpliedPrice}
Mark Britten-Jones and Anthony Neuberger.
\newblock Option prices, implied price processes, and stochastic volatility.
\newblock {\em The Journal of Finance}, 55(2):839--866, 2000.

\bibitem{LevyModels}
Peter Carr, H{\'e}lyette Geman, Dilip~B. Madan, and Marc Yor.
\newblock From local volatility to local {L}\'evy models.
\newblock {\em Quant. Finance}, 4(5):581--588, 2004.

\bibitem{Coquet}
Fran{\c{c}}ois Coquet, Jean M{\'e}min, and Leszek S{\l}omi{\'n}ski.
\newblock On non-continuous {D}irichlet processes.
\newblock {\em J. Theoret. Probab.}, 16(1):197--216, 2003.

\bibitem{Derman}
Emanuel Derman and Iraj Kani.
\newblock Riding on a smile.
\newblock {\em Risk}, 7(2):32--39, February 1994.

\bibitem{Dupire2}
Bruno Dupire.
\newblock Pricing and hedging with smiles.
\newblock In {\em Mathematics of derivative securities (Cambridge, 1995)},
  volume~15 of {\em Publ. Newton Inst.}, pages 103--111. Cambridge Univ. Press,
  Cambridge, 1997.

\bibitem{Follmer2}
H.~F{\"o}llmer.
\newblock Dirichlet processes.
\newblock In {\em Stochastic integrals (Proc. Sympos., Univ. Durham, Durham,
  1980)}, volume 851 of {\em Lecture Notes in Math.}, pages 476--478. Springer,
  Berlin, 1981.

\bibitem{Fontes}
L.~R.~G. Fontes, M.~Isopi, and C.~M. Newman.
\newblock Random walks with strongly inhomogeneous rates and singular
  diffusions: convergence, localization and aging in one dimension.
\newblock {\em Ann. Probab.}, 30(2):579--604, 2002.

\bibitem{HeWangYan}
Sheng{-}wu He, Jia{-}gang Wang, and Jia{-}an Yan.
\newblock {\em Semimartingale theory and stochastic calculus}.
\newblock Kexue Chubanshe (Science Press), Beijing, 1992.

\bibitem{Hobson2}
David~G. Hobson.
\newblock The maximum maximum of a martingale.
\newblock {\em S{\'e}minaire de probabilit{\'e}s de Strsbourg}, 32:250--263,
  1998.

\bibitem{Kellerer}
Hans~G. Kellerer.
\newblock Markov-{K}omposition und eine {A}nwendung auf {M}artingale.
\newblock {\em Math. Ann.}, 198:99--122, 1972.

\bibitem{Lowther4}
George Lowther.
\newblock A generalized backward equation for one dimensional processes.
\newblock Pre-print available as arxiv:0803.3303v2 [math.PR] at arxiv.org,
  August 2008.

\bibitem{Lowther2}
George Lowther.
\newblock Limits of one dimensional diffusions.
\newblock Pre-print available as arXiv:0712.2428v2 [math.PR] at arxiv.org,
  August 2008.

\bibitem{Lowther3}
George Lowther.
\newblock Nondifferentiable functions of one dimensional semimartingales.
\newblock Pre-print available as arXiv:0802.0331v2 [math.PR] at arxiv.org,
  August 2008.

\bibitem{Lowther1}
George Lowther.
\newblock Properties of expectations of functions of martingale diffusions.
\newblock Pre-print available as arXiv:0801.0330v1 [math.PR] at arxiv.org,
  January 2008.

\bibitem{Yor}
Dilip~B. Madan and Marc Yor.
\newblock Making {M}arkov martingales meet marginals: with explicit
  constructions.
\newblock {\em Bernoulli}, 8(4):509--536, 2002.

\bibitem{Rogers}
L.~C.~G. Rogers and David Williams.
\newblock {\em Diffusions, {M}arkov processes, and martingales. {V}ol. 2}.
\newblock Wiley Series in Probability and Mathematical Statistics: Probability
  and Mathematical Statistics. John Wiley \& Sons Inc., New York, 1987.

\bibitem{Rogers1}
L.~C.~G. Rogers and David Williams.
\newblock {\em Diffusions, {M}arkov processes, and martingales. {V}ol. 1}.
\newblock Cambridge Mathematical Library. Cambridge University Press,
  Cambridge, 2000.
\newblock Foundations, Reprint of the second (1994) edition.

\bibitem{Strassen}
V.~Strassen.
\newblock The existence of probability measures with given marginals.
\newblock {\em Ann. Math. Statist}, 36:423--439, 1965.

\bibitem{Stricker1}
C.~Stricker.
\newblock Variation conditionnelle des processus stochastiques.
\newblock {\em Ann. Inst. H. Poincar\'e Probab. Statist.}, 24(2):295--305,
  1988.

\end{thebibliography}
\bibliographystyle{plain}

\end{document}